\documentclass[11pt]{amsart}

\usepackage{amsthm}
\usepackage[leqno]{amsmath}
\usepackage{latexsym,amsfonts,amssymb,amsthm}
\usepackage[all]{xy} \SelectTips{eu}{} \SilentMatrices
\usepackage[urlcolor=blue]{hyperref}
\usepackage{mathtools}
\usepackage{tikz}
\usepackage[utf8]{inputenc}
\usepackage{fullpage}
\usepackage{mathdots}
\usepackage{marvosym}
\usepackage{paralist} 
\usepackage{longtable}
\usepackage{color}
\usetikzlibrary{fit}
\usetikzlibrary{shapes}
\usepackage{tikz-cd}
\numberwithin{equation}{section}
\usepackage{booktabs}
\makeatletter
\renewcommand*\env@matrix[1][*\c@MaxMatrixCols c]{%
  \hskip -\arraycolsep
  \let\@ifnextchar\new@ifnextchar
  \array{#1}}
\makeatother

\newcommand{\s}{\mathsf s}
\newcommand{\mcA}{\mathcal A}
\newcommand{\A}{\mathcal A}
\newcommand{\ZZ}{\mathbb Z}
\newcommand{\NN}{\mathbb N}

\newcommand{\glom}{\mathcal{G}}

\DeclareMathOperator{\stab}{stab}

\theoremstyle{plain}
\newtheorem{theorem}{Theorem}[section]
\newtheorem{corollary}[theorem]{Corollary}
\newtheorem{lemma}[theorem]{Lemma}
\newtheorem{proposition}[theorem]{Proposition}
\newtheorem{question}[theorem]{Question}

\theoremstyle{definition}
\newtheorem{definition}[theorem]{Definition}

\newtheorem{setup}[theorem]{Setup}

\newtheorem{remark}[theorem]{Remark}

\newtheorem{example}[theorem]{Example}

\begin{document}

\title{Divisor Sequences of Atoms in Krull Monoids}

\dedicatory{This paper is dedicated to Roger and Sylvia Wiegand --- mentors, role models, and friends --- on the occasion of their combined 151st birthday. }

\author{Nicolas R.\ Baeth, Terri Bell, Courtney R.\ Gibbons, and  Janet Striuli}

\address{\hspace{-.16in}\makebox{Nicholas R.\ Baeth, Department of Mathematics, Franklin and Marshall College, Lancaster, PA 17604}\newline \texttt{nicholas.baeth@fandm.edu}.}

\address{\hspace{-.16in}\makebox{Courtney R.\ Gibbons, Department of Mathematics and Statistics, Hamilton College, Clinton, NY 13323}\newline \texttt{crgibbon@hamilton.edu}.}

\address{\hspace{-.16in}\makebox{Terri Bell, Mathematics Department, South Puget Sound Community College, Olympia, WA 98512} \newline \texttt{tbell@spscc.edu}.}

\address{\hspace{-.16in}\makebox{Janet Striuli, Department of Mathematics, Fairfield University, Fairfield, CT 06824} \newline
{National Science Foundation, 2415 Eisenhower Avenue, Alexandria, Virginia 22314}\newline
  \texttt{jstriuli@fairfield.edu}.}

\date{November 2019}

\begin{abstract}
The divisor sequence of an irreducible element (\textit{atom}) $a$ of a reduced monoid $H$ is the sequence $(s_n)_{n\in \mathbb{N}}$ where, for each positive integer $n$, $s_n$ denotes the number of distinct irreducible divisors of $a^n$. In this work we investigate which sequences of positive integers can be realized as divisor sequences of irreducible elements in Krull monoids. In particular, this gives a means for studying non-unique direct-sum decompositions of modules over local Noetherian rings for which the Krull-Remak-Schmidt property fails. 
\end{abstract}

\maketitle

\section{Introduction}

When the Krull-Remak-Schmidt theorem holds (for example, for the class of finitely generated modules over a complete local Noetherian ring), direct-sum decomposition of modules is unique. When Krull-Remak-Schmidt fails, one can ask: for each indecomposable module $M$ over a local Noetherian ring $R$ and for each positive integer $n$, how many non-isomorphic indecomposable $R$-modules are isomorphic to direct summands of $\underset{n}{\underbrace{M\oplus \cdots \oplus M}}$? Of course the answer depends on the ring $R$ and the module $M$, but if we fix these two objects and let $n$ vary, what sequence do we get? More generally, we are interested in understanding  which sequences of positive integers arise in this way when  $R$ and $M$ are allowed to vary.

In \cite{W01} R.~Wiegand showed that every finitely generated Krull monoid can be realized as $+(M)$, the semigroup of isomorphism classes of $R$-modules that are direct summands of direct sums of finitely many copies of $M$, for some $R$-module $M$. This motivated the question above, stated in \cite[Section 9]{WW09} in the language of Krull monoids, where the sequences in question are referred to as \emph{divisor sequences}. As mentioned in that paper, Hassler \cite{H04} gave a single example illustrating that not all reasonable sequences can be realized in this manner. The goal of this work is to describe what sequences occur in this way if one begins with an irreducible element in a Krull monoid: 
 \begin{question}[R.~ and S.~Wiegand]\label{thequestion}
What nondecreasing, eventually constant sequences $$(1, \ldots, 1, s_n, s_{n+1}, \ldots)$$ with $s_n\geq 3$ occur as the divisor sequence of an irreducible element in a Krull monoid?
\end{question}
Our main result, given in Theorem \ref{theorem:FTG} and Corollaries \ref{1isthelonliestnumber} and \ref{partialanswer} is summarized here, and gives an answer to the above question. The necessary definitions are given in Section \ref{s:prelim}.

\bigskip

\noindent {\bf Main Result.} Let $\s=(s_n)_n$ and $(t_n)_n$ be  eventually constant nondecreasing integer sequences with $s_1=t_1=1$. 
\begin{enumerate}
\item Given irreducible elements $x$ and $y$ in Krull semigroups $H_1$ and $H_2$ with divisor sequences $(s_n)_n$ and $(t_n)_n$, there is a Krull semigroup $H$ and an irreducible element $z\in H$ with divisor sequence $(s_n+t_n-1)_n$.
\item If $s_{n+1}-s_n\not=1$ for all $n$, then $\s$ is the divisor sequence of an irreducible element in some Krull monoid.
\item If $s_2\geq 3|\{n\colon s_{n+1}-s_n=1\}|+1$, then $\s$ is the divisor sequence of an irreducible element in some Krull monoid. \end{enumerate}

\smallskip

First, in Section \ref{s:prelim}, we will introduce the relevant definitions and notation. In Section \ref{s:examples}, we provide some infinite families of sequences satisfying some natural properties (see Proposition \ref{easy}) that can and cannot occur as divisor sequences of atoms in Krull monoids. In Section \ref{s:torgleandglom}, we give our main results. In particular, we provide a method for glomming together (see Theorem \ref{theorem:FTG}) two Krull monoids with fixed atoms in such a way that one obtains an atom in a new Krull monoid whose divisor sequence can easily be obtained from the divisor sequences of the two fixed atoms. Then, in Corollary \ref{partialanswer}, we show that most sequences satisfying the properties of Proposition \ref{easy} are realizable as the divisor sequence of an atom in a Krull monoid. Finally, in Section \ref{th}, we study the effect of transfer homomorphisms on divisor sequences. 
\bigskip

\section{Preliminaries}\label{s:prelim}

Throughout, $\mathbb N=\{1, 2, 3, \ldots\}$ denotes the set of positive integers and $\mathbb N_0=\mathbb N\cup\{0\}$. We take a \emph{monoid} $H$ to be a commutative, cancellative  (given $a, b, c \in H$, if $ab=ac$, then  $b=c$) semigroup with the unit element $1$. Moreover, we consider only \emph{reduced} monoids (for $a, b \in H$, if $ab=1$, then  $a=b=1$.) Let $H$ be a reduced monoid. For elements $a,b \in H$ we write $a\mid_Hb$ if there exists an element $c \in H$ such that $b=ac$. An element $a\in H\backslash\{1\}$ is an \emph{atom} (or is said to be \emph{irreducible}) if  the equality $a=bc$, for some $b, c \in H$, implies $b=1$ or $c=1$. We denote the set of atoms of $H$ by $\mathcal A(H)$. Given an element  $h \in H$ we will set $\mcA_n (h) $ to be the set $\{a \in \mcA(H) \colon a \mid h^n\}$. Of course one can write a monoid using additive notation, with $0$ the only invertible element and with $a$ an atom if $a=x+y$ implies exactly one of $x$ or $y$ is zero. For convenience, we will often do just that, with the notation just defined adjusted accordingly. We can now define the main object of interest.

\smallskip

\begin{definition}
Let $h$ be an element of a reduced monoid $H$. The \emph{divisor sequence} of $h$ is the sequence $\s(h)=(s_n)_{n \in \mathbb{N}}$ where each $n \in \mathbb{N}$, $s_n$ is the cardinality of the set $\mcA_n(h)$. For convenience, if there is some $N$ such that $s_n=s_{n+1}$ for all $n\geq N$, we write $\s(h)=(s_1, s_2, \ldots, s_{N-1}, \overline{s_N})$.
\end{definition}

\smallskip

We illustrate this concept now with three simple examples.

\smallskip

\begin{example}\label{easyexamples}
\ \\ \vspace{-.125in}
\begin{enumerate}
    \item If $H$ is a unique factorization monoid and $h\in H\backslash\{1\}$, $h$ factors uniquely as $h=a_1^{r_1}\cdots a_m^{r_m}$ for distinct atoms $a_1, \ldots, a_m\in H$ and $r_1, \ldots, r_m\in \mathbb N$. Then $\s(h)=(m, m,\ldots)=(\overline{m})$.
    \item Let $H$ be the additive reduced monoid $H=\left\{\begin{bsmallmatrix}x \\ y\end{bsmallmatrix}\colon x+2y \equiv 0\mod 3\right\}$. The atoms of $H$ are $\alpha=\begin{bsmallmatrix}3 \\ 0\end{bsmallmatrix}$, $\beta=\begin{bsmallmatrix}0 \\ 3\end{bsmallmatrix}$, and $\gamma=\begin{bsmallmatrix}1 \\ 1\end{bsmallmatrix}$. Since $\gamma+\gamma+\gamma=\alpha+\beta$, $\s(\gamma)=(1, 1, \overline{3})$.
    \item Let $x$ be an indeterminate and let $H=\{x^t\colon t\in \{0,4,5,6,8,9,10,\ldots\}\}$, a multiplicative submonoid of $\mathbb Q[x]$. Then $H$ is a reduced monoid with $\mathcal A(H)=\{x^4, x^5, x^6\}$. It is easy to check that $\s(x^4)=(1,1,2,\overline{3})$, $\s(x^5)=(1,\overline{3})$, and $\s(x^6)=(1, 2, \overline{3})$.
\end{enumerate}
\end{example}

\smallskip

\begin{remark}
An atom $a$ of a reduced commutative cancellative semigroup $H$ is a \emph{strong atom} (see \cite{CK12}) if the only atom of $H$ that divides powers of $a$ is $a$ itself. Example \ref{easyexamples}(1) then can be generalized to say that $h$ is a strong atom if and only if $\s(h)=(1,1,1,\ldots)$. Thus, in general, the growth of the sequence $\s(h)$, for an atom $h$, provides a measure of how far $h$ is from being a strong atom.
\end{remark}

\smallskip

In large part because the motivation for studying divisor sequences comes from module theory, we restrict our study to Krull monoids. Therefore, in what follows we will consider only divisor sequences $\s(a)$ with $a$ an atom of a Krull monoid $H$.  A thorough treatment of Krull monoids and their arithmetic can be found in \cite{GHK06}.

\smallskip

Let $H$ and $D$ be two  monoids, a \emph{divisor homomorphism} is a semigroup homomorphism $\varphi: H \rightarrow D$ such that, for $x,y\in H$,  $x\mid_H y $ if and only if $\varphi(x)\mid_K\varphi(y)$.  A monoid $H$ is \emph{Krull} if there is a divisor homomorphism $\varphi: H\rightarrow F$ for some free monoid $F$. 

\smallskip
The proposition below \cite[Proposition 2.7]{WW09} gives equivalent characterizations of Krull semigroups; we will often use statement (2)  when we wish to prove that a semigroup is Krull.

\begin{proposition}[\cite{WW09}, Proposition 2.7]\label{krull}
Let $H$ be a finitely generated reduced monoid. The following are equivalent.
\begin{enumerate}
    \item $H$ is Krull.
    \item $H$ is a positive normal affine semigroup; that is, $H$ is isomorphic to a finitely generated subsemigroup of $\mathbb N_0^t$ for some positive integer $t$ and such that if $g\in \mathbb ZH$ and $Ng\in H$ for some $N\in \NN$, then $g\in H$.
    \item $H$ is isomorphic to a full subsemigroup $H'$ of $\mathbb N_0^s$ for some $s\geq 1$; the inclusion $H'\subseteq \mathbb N_0^s$ is a divisor homomorphism.
    \item $H$ is isomorphic to an expanded subsemigroup $H'$ of $\mathbb N_0^t$ for some $t\geq 1$; $H=\mathbb N_0^t\cap L$ for some $\mathbb Q$-subspace $L$ of $\mathbb Q^t$.
    \item $H\cong G\cap \mathbb N_0^t$ for some $t\geq 1$ and some subgroup $G$ of $\mathbb Z^t$.
\end{enumerate}
\end{proposition}

\smallskip

Let $a$ and $b$ be atoms of a finitely generated reduced Krull monoid $H$. If $m<n$, then it is clear that $b\mid a^m$ implies $b\mid a^n$. Consequently, $\s(a)$ is always a nondecreasing sequence. It is also clear that if $b\mid a$, then $b=a$, and so $\s(a)$ begins with $1$. And, since $H$ is finitely generated, $\mathcal A(H)$ is finite and thus $\s(a)$ eventually stabilizes. Another easy, though less obvious fact, pointed out in \cite[Section 9]{WW09}, is that $2$ does not appear in $\s(a)$ if $a$ is an atom. Indeed, suppose that this is not the case and let $n \in \mathbb{N}$  be the smallest integer such that $a^n$ is divisible by precisely two atoms, $a$ and $b\not=a$. 
Since $b\mid a^n$, $a^n=b^ka^m$ for some positive integerms $m\leq n$ and  $k$. By cancellation, $b^k=a^{n-m}$. 
By Proposition \ref{krull}, we can view $H$ as an additive full submonoid of $\mathbb N^{(t)}$ for some $t$. Using the component-wise partial ordering on $\mathbb N^{(t)}$, $b\leq a$ or $a\leq b$, depending on the values of $k$ and $n-m$. Then in $\mathbb N^{(t)}$, and also in $H$ since $H$ is full in $\mathbb N^{(t)}$, $b\mid a$ or $a\mid b$. Since $a$ and $b$ are irreducible, we have that $a=b$.

The observations above can be summarized in the following, \cite[Section 9]{WW09}:

\smallskip

\begin{proposition}\label{easy}
Let $a$ be an atom of a finitely generated reduced Krull monoid and let $\s(a)=(s_n)_{n\in\mathbb{N}}$. Then
\begin{enumerate}
    \item $s_1=1$,
    \item  $s_n\not=2$ for any $n \in \mathbb{N}$,
    \item $s_{n+1}\geq s_n$, for all $n \in \mathbb{N}$,
    \item there exists $N \in \mathbb{N}$ such that $s_n=s_{n+1}$ for all $n \geq N$.
    \end{enumerate}
\end{proposition}

\smallskip

\begin{remark}
Clearly all but the second condition of Proposition \ref{easy} hold for the divisor sequence of an atom in any commutative cancellative monoid that is finitely generated. However, the second condition need not hold; see Example \ref{easyexamples}(3). Thus the existence of certain divisor sequences can be used as a test to determine that a monoid is not Krull.
\end{remark}

\smallskip

Before closing the section, we  prove a useful generalization of Proposition \ref{easy}(2). This result will be used in Section \ref{s:examples} to show that certain sequences cannot occur as divisor sequences of atoms in Krull monoids.

\smallskip

\begin{lemma}\label{lemma:minimalpower}
Let $H$ be a finitely generated reduced Krull monoid. Let $x,y\in \mathcal A(H)$ and let $m, n \in \NN$ with $m \leq n-1$. If $x^m \mid y^n$ but $x^m \nmid y^{n-1}$, then $x^{m+1} \nmid y^n$.
\end{lemma}

\begin{proof}
As $H$ is Krull we can view (Proposition \ref{krull}) $H$ as a full submonoid of $\NN_0^{(t)}$ for some $t\geq 1$. For the sake of contradiction, assume that $x^{m+1} \mid y^n$. With additive notation in $\NN_0^{(t)}$, it follows that $mx_i \leq ny_i$ holds for all $i\in \{1, \ldots, t\}$,  $(m+1)x_i \leq ny_i$ holds for all $i \in \{1, \ldots, t\}$, and $mx_j > (n-1)y_j$ holds for some $j\in \{1, \ldots, t\}$. Choosing such a $j$, we obtain
 \[(m+1)x_j\leq ny_j=(n-1)y_j+y_j<mx_j+y_j,\]
 and so $x_j<y_j$. This implies $(n-1)y_j<mx_j<my_j$, which lead to a contradiction since $m\leq n-1$.
\end{proof}

\bigskip

\section{Examples and Nonexamples}\label{s:examples}

In this section, we provide examples of Krull monoids and some divisor sequences realized by some of their irreducible elements. As will become clear in Section \ref{s:torgleandglom}, we are primarily interested in divisor sequences with nice features: small initial entries and with increases of small size at desired indices. These  examples (and their nice features) allow us, in Section \ref{s:torgleandglom}, to describe a heuristic for constructing new divisor sequences with desired features.  We also give some examples of sequences that cannot be obtained as divisor sequences of atoms in Krull monoids.

\smallskip

Our first example allows for the construction of a divisor sequence whose only values can be $1$ or  $(n+1)$s, where $n \geq 2$ and where the first $n+1$ can appear as early as the second entry.

\smallskip

\begin{theorem}\label{terri}
Let $n,k\in \NN$ with $n,k\geq 2$. There exists a Krull monoid $H$ and an irreducible element $a\in H$ such that $\s(a)=(s_n)_{n\in \mathbb{N}}$ where 
\[
\begin{cases}
 s_i=1, \qquad \text{for} \;i <k,\\
  s_i=n+1, \quad \text{for} \; i\geq k.
\end{cases} 
\]
\end{theorem}

\begin{proof}
Let $k$ and $n$ be positive integers, and set $N:=\binom{n+k-1}{n-1}$. There are $N$ ordered partitions of $k$ as a sum of $n$ nonnegative integers. Let $A$ be the $N\times n$ matrix whose rows correspond to the $N$ partitions ordered lexicographically. That is, the first row of $A$ is $\mathbf{a}_1=\begin{bmatrix}k & 0 & \cdots & 0\end{bmatrix}$, the second row of $A$ is $\mathbf{a}_2=\begin{bmatrix}k-1 & 1 & 0 & \cdots & 0\end{bmatrix}$, and the $N$th row of $A$ is $\mathbf{a}_N=\begin{bmatrix}0 & \cdots & 0 & k\end{bmatrix}$.
 Denote the columns of $A$ by $\mathbf{x}_1, \ldots , \mathbf{x}_n$, and set $\mathbf{x}_0$ to be the $N$-dimensional tuple $\begin{bmatrix}1 & \cdots & 1\end{bmatrix}^T$. For each $i\in \{0, \ldots , n\}$, denote by $\mathbf{x}_i[j]$ the $j$th coordinate of the vector $\mathbf{x}_i$which corresponds to the $ij$th entry of the matrix $A$. 
 
 Now let $H$ be the submonoid of $\NN_0^{(N)}$ generated by $\mathbf{x}_0, \ldots , \mathbf{x}_n$; in particular $H$ is the set of all vectors in $\NN_0^{(N)}$ of the form $\sum_{i=0}^nc_i\mathbf{x}_i$ with each $c_i\in \NN_0$. By the construction of the matrix $A$, it is clear that $\sum_{i=1}^n\mathbf{x}_i=k\mathbf{x}_0$.  We will show that $\mathbf{x}_i$ is an irreducible element in $H$, for each $i=1, \dots, n$. Then, as $H$ is generated by the $\mathbf{x}_i$ and since $\sum_{i=1}^n\mathbf{x}_i=k\mathbf{x}_0$, $|\mcA(\mathbf{x}_0)_k|= n+1$. Moreover, for each $i\in \{1, \ldots , n\}$, $\mathbf{x}_{i}[j]=k$ for some $j\in \{1, \ldots , N\}$ and thus there are no $c_i\in \NN_0$ so that $\sum_{i=1}^nc_i\mathbf{x}_i=l\mathbf{x}_0$ if $l<k$. Consequently $|\mcA(\mathbf{x}_0)_l|=1$ if $l<k$ and $\s(\mathbf{x}_0)=(1, \ldots , 1, \overline{n+1})$. 

\smallskip

We now show that each $\mathbf{x}_i$ is an irreducible element in $H$ for $i=0, \dots, n$. Since the  vectors $\mathbf{x}_i$ for $i=1, \dots, N$ generate $M$, we need only show that for each $i=1, \dots, n$, $\mathbf{x}_i$ cannot be written as an $\NN_0$-linear combination of the other $\mathbf{x}_j$s. For each $i\in \{1, \ldots , n\}$, there is $j\in \{1,\ldots, N\}$ such that $\mathbf{x_i}[j]=k>1$. Since $\mathbf{x}_0[j]=1$, $\mathbf{x}_0$ is not an $\NN_0$-linear combination of $\mathbf{x}_1, \ldots , \mathbf{x}_n$.  Conversely, for all $i\in \{1, \ldots , n\}$, there is a $j\in \{1, \ldots , N\}$ such that $\mathbf{x}_i[j]=0$ and so if $\mathbf{x}_i=\sum_{j\not=i}c_j\mathbf{x}_j$ with each $c_j\in \NN_0$, $c_0=0$. Moreover, for each $i\in \{1, \ldots, n\}$, there is $j\in \{1, \ldots, N\}$ such that $\mathbf{x}_i[j]=k$ and $\mathbf{x}_l[j]=0$ for all $l\in \{1, \ldots, n\}\backslash\{i\}$. Thus $\mathbf{x}_i=\sum_{j\not=i}c_j\mathbf{x}_j$ with each $c_j\in \NN_0$ is impossible and so each $\mathbf{x}_i$ is irreducible in $M$.

\smallskip

We  finish the proof by showing that $H$ is indeed a Krull monoid. We show that $H$ is an expanded submonoid of $\NN_0^N$, and hence Krull by Proposition \ref{krull}. Let $\mathbf{y}\in \mathbb ZH$ and suppose that $m\mathbf{y}\in H$ for some $m\in \NN$. We must show that $\mathbf{y}\in H$. Since $m\mathbf{y}\in H$, we can write 
\begin{equation}
\label{doug}m\mathbf{y}=\sum_{i=0}^nc_i\mathbf{x}_i,
\end{equation} 
for some $c_i\in \NN_0$. If necessary, we can use the relation $\sum_{i=1}^n\mathbf{x}_i=k\mathbf{x}_0$ to replace $c_0, c_1, \ldots , c_n$ by $c_0+k, c_1-1, \ldots , c_n-1$. Thus we may assume that $c_j=0$ for some $j\in \{1, \ldots , n\}$. Considering only the first coordinate of the $\mathbf{x}_i$'s, that is the first row of the matrix $A$, $\mathbf{x}_0[1]=1$, $\mathbf{x}_1[1]=k$, and $\mathbf{x}_i[1]=0$ for all $i \in \{2, \ldots , n\}$. Thus, considering only the $1$st coordinate, Equation (\ref{doug}) becomes $m\mathbf{y}[1]=c_0+kc_1$. Similarly, considering rows $2, \ldots, n$ of $A$, Equation (\ref{doug}) becomes $m\mathbf{y}[i]=c_0+c_1(k-1)+c_i$ for each $i\in \{2, \ldots, n\}$. Modulo $m$, these $n$ equations become $0 \equiv c_0+c_1k\equiv c_0+c_1(k-1)+c_i\mod m$, whence $c_i\equiv c_1\mod m$ for all $i\in \{1, \ldots , n\}$. Now since $c_j=0$ for some $j$, $c_i\equiv 0 \mod m$ for all $i\in \{1, \ldots, n\}$. But then $c_0\equiv 0\mod m$ as well and Equation (\ref{doug}) becomes \begin{equation}\label{thereturnofdoug} \mathbf{y}=\sum_{i=0}^n\frac{c_i}{m}\mathbf{x}_i\end{equation} where $\frac{c_i}{m}\in \NN_0$ for each $i\in \{0, \ldots, n\}$. Thus $\mathbf{y}\in H$ and $H$ is Krull.
\end{proof}

\smallskip

The next result is obtained in a similar way and allows for the construction of divisor sequences with an increase of size $1$ in a desired entry. There is, however, a catch in that this increase must happen at an odd index and requires earlier entries in the sequence to be sufficiently large. These sequences will be used to build more general sequences in Section \ref{s:torgleandglom}.

\smallskip

\begin{proposition}[\cite{N11}]\label{zach}
Let $n,k\in \NN$ such that $k$ is odd and $k\geq 3$,  and $n\geq 4$. There exists a Krull monoid $H$ and an irreducible element $a\in H$ with $\s(a)=(s_n)_{n \in \NN}$ such that 
\[
\begin{cases}
s_1=1;\\
s_i=n, 1<i<k,\\
s_i=n+1, i\geq k.
\end{cases}
\]
\end{proposition}

\smallskip 

In light of Propositions \ref{terri} and \ref{zach}, it is natural to ask what other sequences $(s_n)_{n\in \NN}$ for which $s_{N+1}-s_N=1$ for some $N\in \NN$, can be realized as divisor sequences of atoms in Krull monoids. The next result gives necessary and sufficient conditions for the realizability of certain such sequences.

\smallskip

\begin{proposition}\label{13334}
Let $\s=(s_i)_{i \in \NN}$ be a nondecreasing sequence of positive integers that eventually stabilizes such that 
\[
\begin{cases}
s_1=1;\\
s_i=3, 1<i<n-1;\\
s_n=4.
\end{cases}
\]

There exists   an irreducible element $x$  in a Krull monoid $H$  such that $\s=\s(x)$ if and only if $n$ is odd.
\end{proposition}
Before we prove the proposition, recall that the block monoid $\mathcal B(\mathbb Z_n, \{1, m\})$ is the submonoid of the free monoid $\mathcal F(\{1,m\})$ containing formal products of $1$s and $m$s whose elements would sum to zero in $\mathbb Z_n$. Each of its elements has the form $1^am^b$, with $0\leq a,b\leq n$ and $a+mb=nk$ for some $k\geq 1$. 
\begin{proof}
Write $n=2m+1$ for a positive integer $m$.  By \cite[Theorem 2.1]{CS03}, $\mathcal B(\mathbb Z_n, \{1, m\})$ has precisely four irreducible elements: $1^n, 1^1m^2, 1^{m+1}m, m^n$. It is then easy to see that  if $\s(1^{m+1}m)= (s_n)_n$, then
\[
\begin{cases}
s_0=1;\\
s_i=3, 1\leq i\leq 2m;\\
s_i=4, i>2m.
\end{cases}
\]
In particular, if $n$ is odd, $\s=\s(x)$ where $x$ is an irreducible element in a Krull monoid $\mathcal B(\mathbb Z_n, \{1, m\})$.

\smallskip

Now suppose that $H$ is a Krull monoid and that $x\in H$ is irreducible with $\s(x)=\s$. Since $|\A_2(x)|=3$ and $H$ is Krull,  we can write $x^2=yz$ for  irreducible elements $y$ and $z$ in $H$. We note that $x^2\not=x^my^n$ for any $m,n\geq 2$ by Lemma \ref{lemma:minimalpower}. We claim that for any $i\in\{2, \ldots, n-1\}$, $y^j\mid x^i$ if and only if $z^j\mid x^i$; which implies that the only factorizations of $x^i$ have the form $x^{i-2j}y^jz^j$ for $0 \leq j\leq i/2$. If $x^i=x^ay^bz^c$, then $x^{i-a}=y^bz^c$, so we may assume that $a=0$; that is, $x^i=y^bz^c$. Write $i=2d+\epsilon$ where $\epsilon\in \{0,1\}$. Then $x^\epsilon y^dz^d=y^bz^c$. If $d\leq b,c$, then $x^\epsilon=y^{b-d}z^{c-d}$ which is a contradiction unless $b=c=d$. If $d\geq b,c$, then $x^\epsilon y^{d-b}z^{d-c}=1$ which is a contradiction unless $\epsilon=1$ and $b=c=d$. Thus, without loss of generality, we may assume that $c\leq d\leq b$. Now $x^{\epsilon}z^{d-b}=y^{b-d}$, and so $x^{2d-2c+\epsilon}=x^{\epsilon}z^{d-c}y^{d-c}=y^{b-c}$, which since $H$ is Krull, is a contradiction unless $b=c=d$.

\smallskip

Suppose, for the sake of contradiction, that $n=2k$ is even. Then $x^n=y^kz^k$. Since $|\A_n(x)|=4$ and $H$ is Krull, $x^n=wx^ay^bz^c$ for $a,b,c\geq 0$. Again by Lemma \ref{lemma:minimalpower}, the exponent on $w$ cannot be larger than $1$ since $w\nmid x^{n-1}$. Since $|\A_i(x)|<4$ for $i<n$, $a=0$. Without loss of generality, assume $b\geq c$, so that $x^{n+1}=wy^{b-c}y^cz^c=wy^{b-c}x^{2c}$. Then $x^{n+1-2c}=wy^{b-c}$, which is impossible unless $c=0$. Thus the only factorization of $x^n$ that involves $w$ is $x^n=wy^b$ for $b\geq 1$. Now $wy^b=y^kz^k$. Since $y^k \mid x^n$ and $y^k \nmid x^{n-1}$, $y^{k+1}\nmid x^n$ by Lemma \ref{lemma:minimalpower}. Thus $b \leq k$ and so $wy^b=y^kz^k$ becomes $w=y^{k-b}z^k$, a contradiction. Therefore $n$ must be odd if $\s=\s(x)$ for an irreducible $x$ in a Krull monoid $H$.
\end{proof}

\smallskip

The next proposition and example give a partial generalization of Proposition \ref{13334}.

\smallskip

\begin{proposition}\label{second134}
Let $m<n$ be positive integers and let $\s=(s_i)_{i\in \NN}$ be a nondecreasing sequence of positive integers that eventually stabilizes such that
\[
\begin{cases}
s_i=1, 1\leq i<m;\\
s_i=3, m\leq i<n;\\
s_n=4.
\end{cases}
\]
If $\s=\s(x)$ where $x$ is an irreducible element in a Krull monoid $H$, then $m\nmid n$.
\end{proposition}

\begin{proof}
Suppose that $H$ is a Krull monoid and that $x\in H$ is irreducible with $\s(x)=\s$. Write $n=mq+r$ with $q\geq 1$ and $0\leq r < m$. Since $|\A_i(x)|=3$ for all $i\in \{m, \ldots, n-1\}$, an argument similar to that in the proof of Proposition \ref{13334} shows that the only factorizations of $x^i$, for $i\in \{m, \ldots, n-1\}$, have the form $x^{i-mj}y^jz^j$. Since $|\A_n(x)|=4$, an argument similar to that in the proof of Proposition \ref{13334} gives the existence of an irreducible element $w$ such that $x^n=wy^b$ with $b\geq 1$ is the only factorization of $x^n$ involving $w$. If $m\mid n$, then $r=0$ and $wy^b=x^n=y^qz^q$. Since $y^q\mid x^n$ and $y^q\nmid x^{n-1}$, Lemma \ref{lemma:minimalpower} implies that $y^{q+1}\nmid x^n$, so that $b\leq q$. Thus $w=y^{q-b}z^q$, a contradiction. Thus $m\nmid n$.
\end{proof}

\smallskip

We conclude this section with one final example illustrating the existence of divisor sequences of the form $(1, 1, \ldots, 1, 3, 3, \ldots, 3, \overline{4})$ provided the positions of the first $3$ and first $4$ do not violate the conditions of Proposition \ref{second134}.

\smallskip

\begin{example}
Let $m$ be an odd integer and set $n=(3m+1)/2$. Note that $\gcd(m,n)=1$ and that $n<2m$. Let $H=\mathcal B(\mathbb Z_n, \{1,m\})$. The elements of $H$ all have the form $1^am^b$ with $0\leq a,b\leq n$ and $a+mb=nk$ for some $k\geq 1$. By \cite[Theorem 2.1]{CS03}, the irreducible elements of $H$ are: $m^n, 1m^3, 1^{n-m}m, 1^n$. It is then easy to check that $\s(1m^3)=(\underbrace{1, \ldots, 1}_{n-m-1}, \underbrace{3, \ldots, 3}_{m}, \overline{4})$ and $\s(1^{n-m}m)=(1, 1, \underbrace{3, \ldots, 3}_{n-3}, \overline{4})$.
\end{example}

\bigskip

\section{Glomming and its Consequences}\label{s:torgleandglom}

Suppose that we have constructed two Krull monoids $H_1$ and $H_2$ and irreducible elements $x\in H_1$ and $y\in H_2$ with divisor sequences respectively equal to $\s(x)$ and $\s(y)$.  The goal of this section is to glom together $H_1$ and $H_2$ to form a new Krull monoid $\glom(H_1, H_2)$ containing an irreducible element $z$ with divisor sequence $\s(z)$ that can be computed from $\s(x)$ and $\s(y)$.

To this end, we consider first differences of divisor sequences. Given a divisor sequence $\s(h)=(s_1, s_2, \ldots)$, define the \emph{first difference sequence}  $\partial \s = [s_{i+1}-s_i]_{i\geq 1}$. In this section we will use the notation from the following setup.

\smallskip

\begin{setup}\label{glom}
Let $H_1$ be a Krull monoid so that the inclusion $H_1 \subseteq \NN_0^{(t_1)}$ is a divisor theory. That is, whenever $h\in \ZZ H_1$ and whenever $B\in \NN$ with $Bh\in H$, then $h\in H_1$. Let $\mathbf{x}_0,\mathbf{x}_1, \ldots, \mathbf{x}_m$ form a set of irreducible elements of $H_1$ so that the only atoms of $H_1$ that appear in factorizations of sums of $\mathbf{x}_0$ are the $\mathbf{x}_i$. Similarly, let $H_2$ be a Krull monoid with divisor theory $H_2\subseteq \NN_0^{(t_2)}$, and let $\mathbf{y}_0, \mathbf{y}_1, \ldots , \mathbf{y}_n$ form a set of irreducible elements so that the only atoms of $H_2$ that appear in factorizations of sums of $\mathbf{y}_0$ are the $\mathbf{y}_j$. Without loss of generality, we may assume that $\mathbf{x}_{0}[k]\not=0$ for all $k=1, \dots, t_1$ (and similarly, we may assume $\mathbf{y}_{0}[k]\not=0$ for all indices $k=1, \dots, t_2$). Indeed, if $\mathbf{x}_{0}[j]=0$ for some $j$, then $\mathbf{x}_{i}[j]=0$ for all $i=1, \dots, m$. In this situation, one could simply replace $H_1$ with its projection ignoring the $j$th component.

\smallskip

 We now glom together $H_1$ and $H_2$ to form the new monoid $\glom(H_1, H_2) \subseteq \mathbb N_0^{(t_1)}\otimes_{\mathbb N_0}\mathbb N_0^{(t_2)}$. Set

  \begin{eqnarray*}
  \glom(H_1, H_2)&=&\langle \mathbf{x}_0\otimes \mathbf{y}_0, \{\mathbf{x}_i\otimes \mathbf{y}_0\colon 1\leq i\leq m\}, \{\mathbf{x}_0\otimes \mathbf{y}_j\colon 1\leq j\leq n\}\rangle \\ 
  &=&\left\{a(\mathbf{x}_0\otimes \mathbf{y}_0)+\sum_{i=1}^m b_i(\mathbf{x}_i\otimes \mathbf{y}_0)+\sum_{j=1}^n c_j(\mathbf{x}_0\otimes \mathbf{y}_j)\colon a, b_i, c_j\in \mathbb N_0\right\}\subseteq \NN_0^{(t_1)}\otimes_{\NN_0}\NN_0^{(t_2)}.
  \end{eqnarray*} 
 That is, $\glom(H_1, H_2)\subseteq \NN_0^{(t_1)}\otimes_{\NN_0}\NN_0^{(t_2)}\cong \NN_0^{(t_1+t_2)}$ is contained in a tensor product $H_1\otimes_{\NN_0}H_2$ of two commutative Krull semigoups. We note that tensor products of semigroups always exist and satisfy the usual universal properties; see \cite{G69}. 
 \end{setup}
 
 \smallskip
 
 We can now state the Fundamental Theorem of Glomming.

\smallskip

\begin{theorem}[Fundamental Theorem of Glomming]\label{theorem:FTG}
Let $H_1\subseteq \NN_0^{(t_1)}$ and $H_2\subseteq \NN_0^{(t_2)}$ be as in Setup \ref{glom} with $\s(\mathbf{x}_0)=(s_1, s_2, s_3, \ldots )$ and $\s(\mathbf{y}_0)=(t_1, t_2, t_3, \ldots)$. Then  the following hold:
\begin{enumerate}
\item $\mathbf{x}_0\otimes\mathbf{y}_0$ is an irreducible element in $\glom(H_1, H_2)$,
\item  $\glom(H_1, H_2)$ is a Krull semigroup, and 
\item $\s(\mathbf{x}_0\otimes\mathbf{y}_0)=(s_1+t_1-1, s_2+t_2-1, s_3+t_3-1, \ldots)$. 
\end{enumerate}
In particular, as sequences, $\partial \s(\mathbf{x}_0\otimes\mathbf{y}_0)=\partial \s(\mathbf{x}_0)+\partial \s(\mathbf{y}_0)$.
\end{theorem}

\smallskip

In light of Proposition \ref{krull}, given any two Krull semigroups $H_1$ and $H_2$ and irreducible elements $x\in H_1$ and $y\in H_2$, we can view $x$ as $\mathbf{x}_0$ and $y$ as $\mathbf{y}_0$ in Setup \ref{glom}. Consequently, there is an irreducible element $z$ in some other Krull semigroup so that $\partial \s(z)=\partial \s(x)+\partial \s(y)$. We will use this observation to build new divisor sequences at the end of this section.

\smallskip

\begin{proof}
First observe that if $a\mathbf{x}_0=\sum_{i=1}^mb_i\mathbf{x}_i$ for some $a,b_i\in \NN_0$, then 

$$a(\mathbf{x}_0\otimes \mathbf{y}_0)=\left(a\mathbf{x}_0\right)\otimes \mathbf{y}_0=\left(\sum_{i=1}^mb_i\mathbf{x}_i\right)\otimes \mathbf{y}_0=\sum_{i=1}^mb_i(\mathbf{x}_i\otimes \mathbf{y}_0).$$ 

Similarly, if $a\mathbf{y}_0=\sum_{j=1}^nc_i\mathbf{y}_j$ for some $a,c_j\in \NN_0$, then 

$$a(\mathbf{x}_0\otimes \mathbf{y}_0)=\mathbf{x}_0\otimes \left(a\mathbf{y}_0\right)=\mathbf{x}_0\otimes\left(\sum_{j=1}^nc_j\mathbf{y}_j\right)=\sum_{j=1}^nc_j(\mathbf{x}_0\otimes \mathbf{y}_j).$$ 

Thus relations from $H_1$ and $H_2$ are preserved in $\glom(H_1, H_2)$. We now show that no additional relations among the generators of $\glom(H_1, H_2)$ can occur. Suppose that 

$$a(\mathbf{x}_0\otimes \mathbf{y}_0)=\sum_{i=1}^mb_i(\mathbf{x}_i\otimes \mathbf{y}_0)+\sum_{j=1}^nc_j(\mathbf{x}_0\otimes \mathbf{y}_j),$$ 

\noindent for some $a,b_i, c_j\in \NN_0$. We now construct a bilinear map $\beta\colon H_1\times H_2\rightarrow \NN_0^{(t_1t_2)}$. For each $i=0, \dots, m$ consider the elements in $\overline{\mathbf{x}}_i \in \NN_0^{(t_1t_2)}$ formed by repeating $t_2$ times the element $\mathbf{x}_i$. That is, if  we denote the $j$th entry of $\mathbf{x}_i$ by $\mathbf{x}_i[j]$  then $ \overline{\mathbf{x}}_i[k]= \mathbf{x}_i[h]$   for each $k=1, \dots, t_1t_2$  where $h$ is the unique positive integer such that $k=qt_1+h$ for some $q\in \mathbb N_0$.

 Similarly, for each $\mathbf{y}_i$, one can build $\overline{ \mathbf{y}}_i\in \NN_0^{(t_1t_2)}$ by repeating $t_1$ times each of the entries of $\mathbf{y}_i$. Now define $\beta(\mathbf{x}_i, \mathbf{y}_j)=\overline{\mathbf{x}}_i\ast\overline{\mathbf{y}}_j$ where $\ast$ denotes the standard component-wise multiplication of vectors in $\NN_0^{(t_1t_2)}$. 
 
 Since $\beta$ is bilinear, by the universal property of the tensor product, $\beta$ will take the same value on the elements $(a\mathbf{x}_0, \mathbf{y}_0)$ and $(\mathbf{x}_0+ \sum_{i=1}^mb_i\mathbf{x}_i,\mathbf{y}_0+\sum_{j=1}^nc_j \mathbf{y}_j)$. In particular,
  the following equality holds
 $$(a-1)(\overline{\mathbf{x}}_0\ast\overline{\mathbf{y}}_0)=\overline{\mathbf{y}}_0\ast\sum_{i=1}^mb_i\overline{\mathbf{x}}_i+\overline{\mathbf{x}}_0\ast\sum_{j=1}^nc_j\overline{\mathbf{y}}_j+\sum_{(i,j)\in [1,m]\times[1,n]}b_ic_j\overline{\mathbf{x}}_i\ast\overline{\mathbf{y}}_j.$$

 Then, for each $k\in [1,t_1t_2]$, 
 
 $$(a-1)\overline{\mathbf{x}}_{0}[k]\overline{\mathbf{y}}_{0}[k]=\overline{\mathbf{y}}_{0}[k]\sum_{i=1}^mb_i\overline{\mathbf{x}}_{i}[k]+\overline{\mathbf{x}}_{0}[k]\sum_{j=1}^nc_j\overline{\mathbf{y}}_{j}[k]+\sum_{(i,j)\in [1,m]\times[1,n]}b_ic_j\overline{\mathbf{x}}_{i}[k]\overline{\mathbf{y}}_{j}[k].$$

This implies that $(a-1)\overline{\mathbf{x}}_{0}[k]\geq \sum_{i=1}^mb_i\overline{\mathbf{x}}_{i}[k]$ for each $k\in [1,t_1t_2]$ and therefore $(a-1)\mathbf{x}_{0}[k]\geq \sum_{i=1}^mb_i\mathbf{x}_{i}[k]$ for each $k\in [1,t_1]$. Consequently, $(a-1)\mathbf{x}_0\geq \sum_{i=1}^mb_i\mathbf{x}_i$ in $\NN_0^{(t_1)}$. Since $H_1$ is Krull, $(a-1)\mathbf{x}_0\geq \sum_{i=1}^mb_i\mathbf{x}_i$ in $H_1$ as well. That is, there is some factorization of $(a-1)\mathbf{x}_0$ in $H_1$ involving all $\mathbf{x}_i$ such that $b_i>0$. Similarly, there is some factorization of $(a-1)\mathbf{y}_0$ in $H_2$ involving all $\mathbf{y}_j$ such that $c_j>0$. Thus no new relations are found in $\glom(H_1, H_2)$.

\smallskip

We now show that each of the generators of $\glom(H_1, H_2)$ is in fact irreducible. 

Suppose that $\mathbf{x}_i\otimes \mathbf{y_0}=\mathbf{x}_k\otimes \mathbf{y}_0+\sum_{u=1}^n\mathbf{w}_u\otimes \mathbf{z}_u$ for some $i\not=k$ and some $\mathbf{w}_u\in H_1$ and $\mathbf{z}_u\in H_2$. Then, applying the bilinear map $\beta$ defined above, $\overline{\mathbf{x}}_i\ast\overline{\mathbf{y}}_0\leq \overline{\mathbf{x}}_k\ast\overline{\mathbf{y}}_0$ in $\NN_0^{(t_1t_2)}$. But then $\overline{\mathbf{x}}_{i}[l]\overline{\mathbf{y}}_{0}[l]\leq \overline{\mathbf{x}}_{k}[l]\overline{\mathbf{y}}_{0}[l]$ for all $l\in [1,t_1t_2]$ and so $\mathbf{x}_{i}[l]\leq \mathbf{x}_{k}[l]$ for all $l\in [1,t_1]$. Therefore $\mathbf{x}_i\leq \mathbf{x}_k$ in $\NN_0^{(t_1)}$. Since $H$ is Krull, $\mathbf{x}_i\leq \mathbf{x}_k$ in $H_1$ as well, contradicting the fact that $\mathbf{x}_i$ and $\mathbf{x}_k$ are irreducible in $H_1$. Similarly, no $\mathbf{x}_0\otimes \mathbf{y}_j$ is comparable to any other $\mathbf{x}_0\otimes \mathbf{y}_k$ in $\glom(H_1, H_2)$.

 Finally, suppose that $\mathbf{x}_i\otimes \mathbf{y}_0=\mathbf{x}_0\otimes \mathbf{y}_j+\sum_{u=1}^n\mathbf{w}_u\otimes \mathbf{z}_u$ for some $i\in [1,m]$, $j\in [1,n]$, $\mathbf{w}_u\in H_1$, and $\mathbf{z}_u\in H_2$. Define, for each index $k\in [1,t_1]$ and $l\in [1,t_2]$, $\gamma_{k,l}\colon H_1\times H_2\rightarrow \NN_0$ by $\gamma_{k,l}(\mathbf{x}_i, \mathbf{y}_j)=\mathbf{x}_{i}[k]\mathbf{y}_{j}[l]$. For each $k$ and $l$, this  is  a bilinear map and so $\mathbf{x}_{i}[k]\mathbf{y}_{0}[l]\geq \mathbf{x}_{0}[k]\mathbf{y}_{j}[l]$ for each pair $(k,l)$. But, since $\mathbf{x}_0$ and $\mathbf{x}_i$ are incomparable in $H_1$ and since $\mathbf{y}_0$ and $\mathbf{y}_j$ are incomparable in $H_2$, we can choose $k$ and $l$ so that $\mathbf{x}_{0}[k]>\mathbf{x}_{i}[k]$ and $\mathbf{y}_{0}[l]<\mathbf{y}_{j}[l]$, contradicting $\mathbf{x}_{i}[k]\mathbf{y}_{0}[l]\geq \mathbf{x}_{0}[k]\mathbf{y}_{j}[l]$. This shows that no $\mathbf{x}_i\otimes \mathbf{y}_0$ is divisible by any $\mathbf{x}_0\otimes \mathbf{y}_j$. A similar argument shows that no $\mathbf{x}_0\otimes \mathbf{y}_j$ is divisible by any $\mathbf{x}_i\otimes \mathbf{y}_0$ and thus all of the given generators of $\glom(H_1, H_2)$ are irreducible.

\smallskip

Finally, we show that $\glom(H_1, H_2)$ is Krull. Suppose that $\delta\in \mathbb Z\glom(H_1, H_2)$. By considering the forms of the generators of $\glom(H_1, H_2)$, $\delta=\mathbf{h}_1\otimes \mathbf{y}_0+\mathbf{x}_0\otimes \mathbf{h}_2$ for some $\mathbf{h}_1\in \mathbb ZH_1$ and $\mathbf{h}_2\in \mathbb ZH_2$. If $N\in \NN_0$ and $N\delta\in \glom(H_1, H_2)$, $N\mathbf{h}_1\otimes \mathbf{y}_0+\mathbf{x}_0\otimes N\mathbf{h}_2\in \glom(H_1,H_2)$. We claim that $N\mathbf{h}_1\in H_1$ and $N\mathbf{h}_2\in H_2$. Assuming this claim, since each of $H_1$ and $H_2$ is Krull, $\mathbf{h}_1\in H_1$ and $\mathbf{h}_2\in H_2$, whence $\delta\in \glom(H_1,H_2)$. Thus the new semigroup $\glom(H_1,H_2)$ is Krull. We now show that $N\mathbf{h}_1\in H_1$ and $N\mathbf{h}_2\in H_2$. Because $N\delta\in \glom(H_1, H_2)$, we can write $N\delta=\mathbf{k}_1\otimes \mathbf{y}_0+\mathbf{x}_0\otimes \mathbf{k}_2$ with $\mathbf{k}_i\in H_i$. That is, $(N\mathbf{h}_1-\mathbf{k}_1)\otimes \mathbf{y}_0+\mathbf{x}_0\otimes (N\mathbf{h}_2-\mathbf{k}_2)=0$ in $\mathbb Z\glom(H_1, H_2)\subseteq \mathbb Z^{(t_1+t_2)}$. Consider the bilinear map $\mathbb Z^{(t_1)}\times \mathbb Z^{(t_2)}\rightarrow \mathbb Z^{(t_1!t_2!)}$ defined by taking the product of each entry of $\mathbb Z^{(t_1)}$ with each entry of $\mathbb Z^{(t_2)}$ as the components of $\mathbb Z^{(t_1)}$ and $\mathbb Z^{(t_2)}$ sequence through each of the $t_1!\times t_2!$ possible permutations. Then, for each $l_1\in [1,t_1]$ and $l_2\in [1,t_2]$, $(Nh_{1l_1}-k_{1l_1}+x_{0l_1})(Nh_{2l_2}-k_{2l_2}+y_{0l_2})=0$. Since we have considered all permutations of the entries of these vectors, without loss of generality $Nh_{1l_1}-k_{1l_1}+x_{0l_1}=0$ for all $l_1\in [1,t_1]$ and thus $N\mathbf{h}_1-\mathbf{k}_1+\mathbf{x}_0=0$. But then $N\mathbf{h}_1=\mathbf{k}_1-\mathbf{x}_0$, and since $\mathbf{k}_1, \mathbf{x}_0\in H_1$ and $H_1$ is Krull, $N\mathbf{h}_1\in H_1$ as well. Also, since $-\mathbf{x}_0=N\mathbf{h}_1-\mathbf{k}_1$, $(N\mathbf{h}_1-\mathbf{k}_1)\otimes \mathbf{y}_0+\mathbf{x}_0\otimes (N\mathbf{h}_2-\mathbf{k}_2)=0$ becomes $\mathbf{x}_0\otimes (N\mathbf{h}_2-\mathbf{k}_2-\mathbf{y}_0)=0$. Since $\mathbf{x}_0$ is nonzero in each component, using the bilinear map from above gives $N\mathbf{h}_2-\mathbf{k}_2-\mathbf{y}_0=0$. Since $\mathbf{k}_2, \mathbf{y}_0\in H_2$ $N\mathbf{h}_2=\mathbf{k_2}+\mathbf{y}_0\in H_2$ as well. This proves the claim, finishing the proof.

\smallskip

Because all relations from $H_1$ and $H_2$ are preserved in $\glom(H_1, H_2)$ and as no new relations are created in this tensor product, $\mathbf{x}_i\otimes \mathbf{y}_0$ is a divisor of $a(\mathbf{x}_0\otimes \mathbf{y}_0)$ only when $\mathbf{x}_i$ is a divisor of $a\mathbf{x}_0$ in $H_1$ and $\mathbf{x}_0\otimes \mathbf{y}_j$ is a divisor of $a(\mathbf{x}_0\otimes \mathbf{y}_0)$ only when $\mathbf{y}_j$ is a divisor of $a\mathbf{y}_0$ in $H_2$. Consequently, $\s(\mathbf{x}_0\otimes\mathbf{y}_0)=(s_1+t_1-1, s_2+t_2-1, s_3+t_3-1, \ldots)$ in $\glom(H_1, H_2)$ when $\s(\mathbf{x}_0)=(s_1, s_2, s_3, \ldots)$ in $H_1$ and $\s(\mathbf{y}_0)=(t_1, t_2, t_3, \ldots)$ in $H_2$.
\end{proof}

\bigskip

We finish the section by compiling  a list of results that are consequences of the examples from Section \ref{s:examples} and  Theorem \ref{theorem:FTG}. The first result of this kind, Corollary \ref{1isthelonliestnumber}, combines Theorems \ref{terri} and \ref{theorem:FTG}, showing that sequences with only jumps of size $2$ and larger are realizable as divisor sequences of irreducible elements in Krull monoids. Then, in Corollary \ref{partialanswer}, we give a heuristic for building divisor sequences involving jumps of size $1$.

\smallskip

\begin{corollary}\label{1isthelonliestnumber}
Let $\s=(s_n)_n$  be a nondecreasing eventually constant sequence of positive integers such that  such that $s_1=1$ and $s_{i+1}-s_i\in \NN\backslash\{1\}$ for all $i\geq 1$. There exists a Krull monoid $H$ and an irreducible element $a\in H$ such that $\s(a)=\s$.
\end{corollary}

\begin{proof}
Let $\s=(s_1, s_2, \ldots)$ be as in the statement of the corollary.  We will prove the corollary by induction on $i$, in particular we will construct an irreducible element $a^i$ in a Krull monoid $H_i$ such that  if $\s(a^i) = (s_1, \dots, s_i, 0, \dots )$. A restatement of Theorem \ref{terri} in terms of the sequences of first differences, says that if $n,k\in \NN$ with $n\geq 2$ and $k\geq 1$, there exists a Krull monoid $H$ and an irreducible element $a\in H$ such that $\partial\s(a)=[0, \ldots , 0, n, 0, \ldots]$ with the $n$ appearing in the $k$th position. Thus there is a Krull monoid $H_1$ and an irreducible element $a_1\in H_1$ such that the$\partial\s(a_1)=[s_2-1, 0, \ldots, 0]$, proving the induction statement for $i=1$. Assume now that for some $i\geq 1$, there is a Krull monoid $H_i$ and an irreducible element $a^i\in H_i$ such that $\partial\s(a^i)=[s_2-1, \ldots, s_{i+1}-s_i, 0, \ldots]$. Again, by Theorem \ref{terri}, there is a Krull monoid $H'$ and an irreducible element $a'\in H'$ such that $\partial\s(a')=[0, \ldots , 0, s_{t+2}-s_{t+1}, 0, \ldots]$ where the $s_{t+2}-s_{t+1}$ occurs in the $t+1$st position. By the Fundamental Theorem of Glomming, $H_{i+1}=\glom(H_i, H')$ is a Krull monoid and there is an irreducible element $a^{i+1}\in H_{i+1}$ such that $\partial\s(a^{t+1})=\partial\s(a^i)+\partial\s(a')=[s_2-1, \ldots, s_{i+1}-s_i, s_{i+2}-s_{i+1}, 0, \ldots]$. 
By induction, and since $\s$ is eventually constant, there is a Krull monoid $H$ and an irreducible element $a\in H$ such that $\partial\s(a)=[s_2-1, s_3-s_2, \ldots, 0, 0, \ldots]$, whence $\s(a)=\s$.
\end{proof}

\smallskip

In a similar manner, we can combine the Fundamental Theorem of Glomming (Theorem \ref{theorem:FTG}), Proposition \ref{zach}, Proposition \ref{13334} and Corollary \ref{1isthelonliestnumber} to obtain a large family of sequences that can be realized as divisor sequences of atoms in Krull monoids. To simplify the statement of the corollary, we ignore divisor sequences of strong atoms.

\smallskip

\begin{corollary}\label{partialanswer} Let $\s= (s_n)_n$ be a nondecreasing  eventually constant sequence of natural numbers such that $s_1 = 1$ and $s_2 \geq 3$. Let $J = \{j_1, \ldots, j_k\}$ be the indices for which $s_{j_k+1}-s_{j_k}=1$. Partition $J$ by $J_e = \{j_\ell \colon j_\ell \text{ is even}\}$ and $J_o = \{j_\ell \colon j_\ell \text{ is odd}\}$.  Let $k_e = |J_e|$ and $k_o = |J_o|$. If $s_2 \geq 2k_e+3k_o+1$, then $\s$ can be realized as the divisor sequence of an irreducible element in a Krull monoid.
\end{corollary}

\begin{proof}
Write $(s_n)_n$ as  the sum of two sequences $(t_n)_n$ and $(u_n)_n$  such that 

$$
u_i=\begin{cases}
1, \quad\qquad\quad\;\;\; i\in J \\ 
2k_e+3k_0, \quad\; i=2 \\
 0, \quad\qquad\quad\;\;\; \text{otherwise}
 \end{cases}.$$ 
 If $i>2$ is even then by Proposition \ref{13334} there exists a Krull monoid $H_e^i$ and an irreducible element $x_e^i\in H_e^i$ with $\s(x_e^i)=(1, 3, \ldots, 3, \overline{4})$ where the first $4$ is in the $i$th spot. Similarly, if $i>2$ is odd then by Proposition \ref{zach}there exists some Krull monoid $H_0^i$ and an irreducible element $y_o^i\in H_e^i$ with $\s(y_o^i))=(1, 4, \ldots, 4, \overline{5})$ where the first $5$ is in the $i$th spot. Now $t_{i+1}-t_i\in \NN_0\backslash\{1\}$ for each $i$ and so by Corollary \ref{1isthelonliestnumber} there is a Krull semigroup $H$ and an irreducible element $z\in H$ such that $\s(z)=(t_n)$. By Theorem \ref{theorem:FTG}, we may combine  $H$ along with all of the $H_e^i$ and $H_o^i$ to obtain a new Krull semigroup $K$ and an irreducible element $w\in K$ with $\s(w)=\s$.
\end{proof}


\smallskip

The result of Corollary \ref{partialanswer} is not optimal in that the values of $s_2$ and $s_3$ are required to be increased whenever there are jumps of size $1$. However, many jumps of size $1$ can occur, even if $s_2$ and $s_3$ are not large. We see this in the final example of this section. 

\smallskip

\begin{example}
Consider the block monoid $\mathcal B(\mathbb Z_n, \{1,m\})$ with $n=km+1$ for some $k\geq 3$. Then $$\s(1m^3)=(\underbrace{1, \ldots, 1}_m, \underbrace{3, \ldots, 3}_m, \underbrace{4, \ldots, 4}_m, \ldots, \overline{k+3}).$$ Similarly, in $\mathcal B(\mathbb Z_n, \{1,m\})$ with $n=km+1$, $m\geq 2$, and $k\geq 1$, $$\s(1^{(k-1)m+1}m)=(1, 3, 4, \ldots,  \underbrace{k+1, \ldots, k+1}_{n-2}, \overline{k+2}).$$
\end{example}

\bigskip

\section{Transfer Homomorphisms}\label{th}

Many results in factorization theory are proved by transferring arithmetic properties from the object of interest to an easier-to-understand object via {\emph transfer homomorphism}.

 A surjective monoid homomorphism $\theta: H\rightarrow K$ is a \emph{transfer homomorphism} of reduced monoids $H$ and $K$ provided $\theta^{-1}(1_K)=\{1_H\}$ and such that whenever $\theta(h)=xy$ for $x,y\in K$, there exist $a\in \theta^{-1}(x)$ and $b\in\theta^{-1}(y)$ such that $ab=h$. 
 
 It is well-known that transfer homomorphisms preserve irreducibility ($h\in \mathcal A(H) \Leftrightarrow \theta(h)\in \mathcal A(K)$) as well as arithmetic invariants such as lengths of factorizations (see \cite[Proposition 3.2.3]{GHK06}). Other important and well-studied invariants in factorization theory, while not preserved, can be controlled with respect to transfer homomorphisms (see, for example \cite[Theorem 3.2.5]{GHK06}). In this section we provide an initial study of how divisor sequences  behave with respect to transfer homomorphisms. The results are given in Propositions \ref{p:th} and \ref{p:bm}. Limitations of Proposition \ref{p:th} are illustrated in Example \ref{e:last}. 
 
 For convenience, we first introduce two new pieces of notation. For an atom $a$ of a finitely generated monoid $H$, denote by $\stab(a)$ the least positive integer $n$ such that $|\mathcal A_k(a)|=|\mathcal A_{k+1}(a)|$ for all $k\geq n$. That is, $\stab(a)$ is the point at which the divisor sequences $\s(a)$ stabilizes. Also, we denote by $\s_i(a)=|\mathcal A_i(a)|$ the $i$th entry in the sequence $\s(a)$; that is, $$\s(a)=(1, \s_2(a), \ldots, \s_{\stab(a)-1}(a), \overline{\s_{\stab(a)}(a)}).$$ We begin with a result about general transfer homomorphisms.

\smallskip

\begin{proposition}\label{p:th}
Let $H$ and $K$ be Krull monoids and let $\theta:H\rightarrow K$ be a transfer homomorphsim. For an element $a\in \mathcal A(H)$, the following statements hold:
\begin{enumerate}
    \item $\stab(\theta(a))\leq\stab(a)$;
    \item for an element $b\in \mathcal A_{\stab(\theta(a))+k}(a)\backslash \mathcal A_{\stab(\theta(a))}(a)$ with $k\geq 0$, then $\theta(b)=\theta(x)$ for some $x\in \mathcal A_{\stab(\theta(a))}(a)$. In particular, an atom that divides $a^{\stab(\theta(a))+k}$ in $H$ is in the fiber of some atom that already divides $a^{\stab(\theta(a))}$.
\end{enumerate}
\end{proposition}

\smallskip

\begin{proof}
For (1), set  $\stab(a)=n$, so that the divisors of $a^k$ are precisely the divisors of $a^{k+1}$ whenever $k\geq n$. Let $x\in \mathcal A(K)$ with $x\mid \theta(a)^{n+1}$. Then $\theta(a^{n+1})=xy$ for some element $y\in K$. As $\theta$ is a transfer homomorphism, there exists an element $b\in \theta^{-1}(x)$ and $c\in \theta^{-1}(y)$ such that $a^{n+1}=bc$ such that  $b$ is irreducible. As $\stab(a)=n$, $b\mid a^n$ and, because $\theta$ is a homomorphism, $x=\theta(b)\mid \theta(a)^n$. We proved that any element that divides $\theta(a)^{n+1}$ also divides $\theta(a)^{n}$, it follows that $\stab(\theta(a))\leq \stab(a)$.

\smallskip

For (2), pick an element  $b\in \mathcal A_{\stab(\theta(a))+k}(a)\backslash \mathcal A_{\stab(\theta(a))}(a)$ for some $k\geq 0$. Then $b$ is an atom of $H$ that divides $a^{\stab(\theta(a))+k}$ but not $a^{\stab(\theta(a))}$. Write $a^{\stab(\theta(a))+k}=bc$ for $c\in H$. Then $\theta(a)^{\stab(\theta(a))+k}=\theta(b)\theta(c)$ and so $\theta(b)\mid \theta(a)^{\stab(\theta(a))+k}$. By definition, $\theta(b)\mid \theta(a)^{\stab(\theta(a))}$ as well and so $\theta(a)^{\stab(\theta(a))}=\theta(b)y$ for some $y\in K$. As $\theta$ is a transfer homomorphism, there is $x\in H$ with $\theta(x)=\theta(b)$ and $z\in \theta^{-1}(y)$ such that $a^{\stab(\theta(a))}=xz$. Moreover, $x\in \mathcal A(H)$ since $\theta$ is a transfer homomorphism and $b\in \mathcal A(H)$.
\end{proof}

\smallskip

\begin{example}\label{e:last}
Let $H=\left\{(x,y,z)\in \NN_0^{(3)}\colon x+y=2z\right\}\subseteq \NN_0^{(3)}$, a half-factorial Krull monoid with $\mathcal A(H)=\{\alpha=(2,0,1), \beta=(0,2,1), \gamma=(1,1,1)\}$. Observe that $\alpha+\beta=\gamma+\gamma$, so $H$ is not factorial. Since $H$ is half-factorial, the length function $\ell: H\rightarrow \mathbb N_0$ is a transfer homomorphism with $\ell(\gamma)=1$. Then $\s(\gamma)=(1, \overline{3})$, yet $\s(\ell(\gamma))=(\overline{1})$. Thus the inequality in Proposition \ref{p:th}(1) can be strict. Moreover, $\ell(\alpha)=\ell(\beta)=\ell(\gamma)=1$ illustrating Proposition \ref{p:th}(2) as well.
\end{example}

\smallskip

We conclude with a result in the more specific setting that $K$ is the associated block monoid of $H$.

\smallskip

\begin{proposition}\label{p:bm}
Let $H$ be a finitely generated reduced Krull monoid, let $K=\mathcal B(G_0)$ be the associated block monoid, and let $\theta:H\rightarrow K$ be the canonical transfer homomorphism.
\begin{enumerate}
\item For any $B\in \mathcal A(\mathcal B(G_0))$, there is $b\in H$ with $\s(b)=\s(B)$.
\item There is a constant $C=C(G, G_0)$ depending only on the class group $G$ of $H$ and the set of primes $G_0$ such that $\s_n(a)\leq 1+(n-1)C$ for each $n$ and for all $i\geq 1$, $$\s_i(\theta(a))\leq \s_i(a)\leq C\s_i(\theta(a)).$$
\end{enumerate}
\end{proposition}

\smallskip

\begin{proof}
With $H$, $K$, and $\theta$ as defined, we have the following commutative diagram
\[\begin{tikzcd}
H \arrow{r}{\subseteq} \arrow[swap]{d}{\theta=\tilde{\theta}|_H} & F=\mathcal F(P) \arrow{d}{\tilde{\theta}} \\
\mathcal B(G_0) \arrow{r}{\subseteq} & \mathcal F(G_0)
\end{tikzcd}
\]
where $F$ is a free monoid with basis $P$, where the horizontal inclusions are divisor theories, and where $\tilde{\beta}: p \mapsto [p]$. 

\smallskip

(1) If $A=[p_1]\cdots [p_t]$ in $\mathcal B(G_0)$, then $a=p_1\cdots p_t$ in $H$ with $\theta(a)=A$. In this case, $\s(a)=\s(A)$. Indeed, if $b$ is an atom of $H$ with $b\mid a^n$ for some $n$, then, in $F$, $b=p_1^{m_1}\cdots p_t^{m_t}$ for some $m_i$. But then $B:=\theta(b)=[P_1]^{m_1}\cdots [P_t]^{m_t}$ is an atom of $\mathcal B(G_0)$ that divides $A^n$. Conversely, if $a=p_1\cdots p_t$ is the factorization of $a\in \mathcal A(H)$ and $b\mid a^n$ for some $n$, $b=p_1^{m_1}\cdots p_t^{m_t}$ for some $m_i$. Now $B:=[p_1]^{m_1}\cdots [p_t]^{m_t}=\theta(b)$ divides $A^n$ where $A=\theta(a)$.

\smallskip

(2) Let $a=[p_1]^{r_1}\cdots [p_t]^{r_t}$ be an atom of $\mathcal B(G_0)$. If $A\in \theta^{-1}(a)$, $$A=p_{11}^{s_{11}}\cdots p_{1k_1}^{s_{1k_1}}\cdot p_{21}^{s_{21}}\cdots p_{2k_2}^{s_{2k_2}}\cdots p_{t1}^{s_{t1}}\cdots p_{tk_t}^{s_{tk_t}}$$ where $[p_{ij}]=[p_{ij'}]$ for all $j,j'$ and $\sum_j s_{ij}=r_i$ for each $i$. If $B$ is an irreducible divisor of $A$, then $b=\theta(B)$ is an irreducible divisor of $a$. Suppose $B\in \mathcal A_n(A)\backslash \mathcal A_{n-1}(A)$. Then $\theta(B)\in \mathcal A_m(a)$ for some $m\leq n$ and if $\theta(B)=[p_1]^{u_1}\cdots [p_t]^{u_t}$, then $B=p_{11}^{v_{11}}\cdots p_{1k_1}^{v_{1k_1}}\cdot p_{21}^{v_{21}}\cdots p_{2k_2}^{v_{2k_2}}\cdots p_{t1}^{v_{t1}}\cdots p_{tk_t}^{s_{tk_t}}$ with $\sum_jv_{ij}=u_i$ for each $i$ and $v_{ij}>(n-1)s_{ij}$ for some $i$ and $j$. Thus, we can bound $|\mathcal A_n(A)\backslash \mathcal A_{n-1}(A)|$ by considering partitions of the $u_i$s. In particular, this upper bound is maximized if $s_{ij}=1$ for all $i$ and $j$. In this case, the number of distinct irreducible divisors in $\mathcal A_n(A)\backslash\mathcal A_{n-1}(A)$ corresponding to a fixed irreducible divisor $b=[p_1]^{u_1}\cdots [p_t]^{u_t}$ is $\prod_{i=1}^tP_{k_i,n}(nu_i)$ where $P_{k,n}(m)$ denotes the number of partitions of $m$ into $k$ parts (allowing parts to be zero) with at least one part of value $n$.

With $\mathsf D(G_0)$ the Davenport constant of $G_0$, we trivially have that $t\leq \mathsf D(G_0)$ and $nu_i\leq n\mathsf D(G_0)$. Thus $$|\mathcal A_n(A)\backslash\mathcal A_{n-1}(A)|\leq \prod_{i=1}^tP_{k_i,n}(nu_i)\leq \mathsf D(G_0)P_{\mathsf D(G_0),n}(n\mathsf D(G_0)).$$ Moreover, by \cite{zoinks}, \begin{equation}\label{wow}\mathsf D(G_0)P_{\mathsf D(G_0),n}(n\mathsf D(G_0))<\frac{\mathsf D(G_0)^{1/4}e^{c\sqrt{n\mathsf D(G_0)}}}{n^{3/4}}e^{-\frac{2\sqrt{n\mathsf D(G_0)}}{c}Li_2\left(e^{-\frac{c(\mathsf D(G_0)+1/2)}{2\sqrt{n\mathsf D(G_0)}}}\right)}\end{equation} where $c=2\sqrt{\zeta(2)}=\pi\sqrt{2/3}$ and $Li_2(x)=\sum_{i=1}^\infty\frac{x^i}{i^2}$ is the dilogarithm function. 

With $C$ the constant on the right side of (\ref{wow}), if $\s(A)=(t_1, t_2, \ldots)$, $t_{n+1}-t_n\leq C$ for all $n\geq 1$. Since $t_1=1$, $t_n\leq 1+(n-1)C$ for all $n\geq 1$. Moreover, since $\theta^{-1}(\mathcal A_n(\theta(a))\backslash \mathcal A_{n-1}(\theta(a))\subseteq \mathcal A_n(a)$, $\s_i(\theta(a))\leq \s_i(a)\leq C\s_i(\theta(a))$.
\end{proof}

\smallskip

We illustrate Proposition \ref{p:bm} by way of a simple example.

\smallskip

\begin{example}
Let $G$ be a finite abelian group and set $H=\Delta(G\times G)=\{(g,g)\colon g\in G\} \subseteq G\times G$. Then $\mathcal B(G)=\{g_1\cdots g_t\in \mathcal F(G)\colon g_1+\cdots + g_t=0\}$ is the block monoid associated to $\mathcal B_H(G\times G)$, the monoid of all formal sequences in $G\times G$ which sum to an element in $H$. Since $(G\times G)/H\cong G$, the map $$\Theta: \mathcal B_H(G\times G)\rightarrow \mathcal B(G)$$ defined by $$(x_1, y_1)\cdots (x_t, y_t)\longmapsto \overline{(x_1-y_1)}\cdots \overline{(x_t-y_t)}$$ is a transfer homomorphism. In particular, for any $g_1^{r_1}\cdots g_t^{r_t}\in \mathcal A(\mathcal B(G))$ and any $h_{ij}\in G$, $$(g_1+h_{11}, h_{11})\cdots (g_1+h_{1r_1}, h_{1r_1})\cdots (g_t+h_{t1}, h_{t1})\cdots (g_t+h_{tr_t}, h_{tr_t})\in \mathcal A(\mathcal B_H(G\times G)).$$

\smallskip

For example, let $G=\mathbb Z_5$ and let $\alpha=13^3$, an irreducible element in $\mathcal B(G)$. Note that $$\mathcal A(\mathcal B(\mathbb Z_5, \{1,3\}))=\{1^5, 1^23, 13^3, 3^5\}.$$ Then: 
\begin{center}
\begin{tabular}{lcl}
$\alpha^2=(13^3)(13^3)=(1^23)(3^5)$ &\hspace{.25in}& $\alpha^3=(13^3)(13^3)(13^3)=(1^23)(13^3)(3^5)$ \\ $\alpha^4=(13^3)^4=(1^23)(13^3)^2(3^5)=(1^23)^2(3^5)^2$ && $\alpha^5=(1^5)(3^5)^3$
\end{tabular}
\end{center}

\smallskip

Thus $\s(\alpha)=(1, 3, 3, 3, \overline{4})$.

\smallskip

Using the correspondence $1\leftrightsquigarrow (1,0)$ and $3^3 \leftrightsquigarrow (3,0)(4,1)(0,2)$, we consider the element $\tilde{\alpha}=(1,0)(3,0)(4,1)(0,2)\in\theta^{-1}(\alpha)$. By definition, $\tilde{\alpha}$ is the only irreducible divisor of $\tilde{\alpha}$. The irreducible divisors of $\tilde{\alpha}^2$ each correspond to one of the three irreducible divisors of $\alpha$ as follows. The irreducible divisors of $13^3$ corresponds to $(1,0)(3,0)(4,1)(0,2)$, $(1,0)(3,0)^2(4,1)$, $(1,0)(3,0)^2(0,2)$, $(1,0)(3,0)(4,1)^2$, $(1,0)(3,0)(0,2)^2$, $(1,0)(4,1)(0,2)^2$, and $(1,0)(4,1)^2(0,2)$. Of these seven divisors, six are not divisors of $\tilde{\alpha}$. Similarly, $1^23$ corresponds to $(1,0)^2(3,0)$, $(1,0)^2(4,1)$, and $(1,0)^2(0,2)$ and $3^5$ corresponds to $(3,0)^2(4,1)^2(0,2)$, $(3,0)^2(4,1)(0,2)^2$, and $(3,0)(4,1)^2(0,2)^2$, giving six more new divisors of $\tilde{\alpha}^2$ that are not divisors of $\tilde{\alpha}$. Thus $|\mathsf A_2(\tilde{\alpha})\backslash \mathsf A_1(\tilde{\alpha})|=12$. Similar computations for the divisors of $\tilde{\alpha}^n$ with $n\in \{3,4,5\}$ show that $\s(\tilde{\alpha})=(1, 13, 25, 31, \overline{35})$.

\smallskip

Choosing different combinations of the elements from $G\times G$ corresponding to $3\in G$, we can obtain different divisor sequences. For example, if we use the correspondence  $3^3 \leftrightsquigarrow (3,0)^2(0,2)$, then $\s((1,0)(3,0)^2(0,2))=(1,7,10,11 ,\overline{13})$. And, if we use the correspondence $3^3 \leftrightsquigarrow (3,0)^3$, then $\s((1,0)(3,0)^3)=(1,3,3,3 ,\overline{4})$.
\end{example}


\end{document}